\documentclass[11pt]{amsart}
\usepackage[T1]{fontenc}

\usepackage{lmodern, amsfonts,amsmath,amstext,amsbsy,amssymb,
amsopn,amsthm,upref,eucal,mathptmx}

\usepackage{mathrsfs}

\RequirePackage{xcolor} % [dvipsnames]
\definecolor{halfgray}
{gray}{0.55}%chapter numbers will be semi
\definecolor{webgreen}
{rgb}{0,0.4,0}
\definecolor{webbrown}
{rgb}{.8,0.1,0.1}
\definecolor{red}
{rgb}{1,0,0}
\usepackage{microtype}

\newcommand \R {{ \mathbb R}}

\newcommand \N {{ \mathbb N}}
\newcommand \T {{ \mathbb T}}
\newcommand \Proj{{\mathbb P}}

\newcommand{\SL}{%
\operatorname{SL}
}

\newtheorem{theorem}{Theorem}[section]
\newtheorem {lemma} [theorem]{Lemma}

\newtheorem{corollary}[theorem]{Corollary}

\newtheorem{conjecture}[theorem]{Conjecture}

\title[Limits of Geodesic Push-Forwards of Horocycle Invariant Measures ]%
{ Limits of Geodesic Push-Forwards  \\ of Horocycle Invariant Measures}

  \author{Giovanni Forni}

\address{Department  of Mathematics\\
  University of Maryland \\
  College Park, MD USA}
  
\email
    {gforni@math.umd.edu}

\date{\today}

\begin{document}

\def\echo#1{\relax}
    
\begin{abstract}
 \begin{sloppypar}
 We prove several general conditional convergence results on ergodic averages  for horocycle and geodesic subgroups
 of any continuous $\SL(2, \R)$-action on a locally compact space. These results are motivated
 by theorems of Eskin, Mirzakhani and Mohammadi on the $\SL(2, \R)$-action on the moduli space of Abelian
 differentials. By our argument we can derive from these theorems an improved version of the ``weak convergence''
 of push-forwards of horocycle measures under the geodesic flow and a  short proof of weaker versions of theorems of Chaika and
 Eskin \cite{CE} on Birkhoff genericity and Oseledets regularity in almost all directions for the Teichm\"uller geodesic flow.
 
  \end{sloppypar}
\end{abstract}
\maketitle

\section{Introduction}
It has been conjectured that push-forwards, under the forward Teichm\"uller  geodesic flow, of ergodic probability  measures
 for the unstable Teichm\"uller  horocycle flow, and similarly of measures uniformly distributed on unstable horocycle arcs 
  or on arcs of the circle action, converge (to an $\SL(2, \R)$-invariant measure).  To the best of our knowledge, W. Veech 
  was the first to ask this question, after his work \cite{V98} on Siegel  measures (now called Siegel--Veech measures).
  
M.~Bainbridge, J.~Smillie and B.~Weiss have proved  this conjecture for certain invariant orbifolds in the stratum $H(1,1)$ of Abelian differentials with two simple zeros on genus $2$ surfaces (see \cite{BSW}, Theorems  1.5 and 12.7).   
  
  The main purpose of this note is to prove that {\it up to removing a set of times of zero upper density}  the general conjecture is in fact a corollary of  results of A.~Eskin, M.~Mirzakhani and A.~Mohammadi \cite{EM}, \cite{EMM}. 
  
 Our argument is base on the idea of lifting family of (probability) measures on a compact space to measures on the space of probability measures, then derive restrictions from the well-known extremal property of ergodic probability measures with respect to the subset of all invariant measures.  

The same argument applies to limits of push-forwards under the Teichm\"uller geodesic flow of the Lebesgue measure on 
Teichm\"uller horocycle orbit segments or on arcs of circle orbits. In particular, our conclusion that push forwards for circle
orbits  converge to an $\SL(2, \R)$-invariant measure after removing a set of times of zero upper density implies,
by the work of A.~Eskin and H.~Masur \cite{EMa},  a correspondingly improved version of the asymptotic for the counting function 
derived in \cite{EMM}, \cite{EM} (see \cite{EM}, Theorem 1.7). 

The first section (\S~\ref{sec:LGP}) of this note is devoted to the proof of the above mentioned results for limit of push forwards of horocycle measures, horocycle and circle arcs.

\smallskip

A similar argument gives a short proof of a weak version of the theorem of J. ~Chaika and A.~Eskin, according to which, for all points
in the moduli space of Abelian differentials, almost all directions are {\it Birkhoff generic} for the Teichm\"uller geodesic flow with respect
to the unique absolutely continuous probability affined measure on the orbifold $\overline{\SL(2, \R) x}$
(see~\cite{CE}, Theorem 1.1).  We are unable to give a proof of the theorem of Chaika and Eskin. In our version, we prove convergence
of ergodic averages outside a subset of times of zero lower density. The second section (\S~\ref{sec:BG}) of this note is devoted to the proof of our partial result on {\it Birkhoff genericity} in almost all directions for general actions of $\SL(2,\R)$.  The third section (\S~\ref{sec:OR}) contains a similar approach to {\it Oseledets regularity} in almost all directions for uniformly Lipschitz irreducible cocycles over $\SL(2,\R)$ actions. Finally, in the last section 
(\S~\ref{sec:horospheres}) we derive 
an equidistribution result for the push-forwards of an arbitrary horospherical leaf under the Teichm\"uller geodesic flow.

\smallskip

In fact, our results are in principle not limited to the action of $\SL(2,\R)$ on the moduli space of Abelian differentials
and hold more generally for general continuous actions on locally compact topological spaces. 

For this reason, we present
below abstract results, which can then be applied to the action on moduli spaces thanks to the celebrated theorems of A.~Eskin, M.~Mirzakhani 
and A.~Mohammadi \cite{EM}, \cite{EMM}.

\medskip  Let
$$
g_t := \begin{pmatrix}e^t &  0 \\ 0 & e^{-t}  \end{pmatrix}\, \quad    h_t := \begin{pmatrix}1 &  t \\ 0 & 1  \end{pmatrix}
\quad \text{ and } \quad r_\theta =  \begin{pmatrix}\cos\theta &  \sin\theta \\ -\sin\theta & \cos \theta \end{pmatrix}
$$
denote the diagonal subgroup (the geodesic flow), the unstable unipotent flow (the unstable horocycle flow)
and the maximal torus (circle) of $\SL(2,\R)$. For all $t, s\in \R$ we have the commutation relation
$$
g_t \circ h_s =     h_{s e^{2t}} \circ  g_t \,.
$$

We will consider below an arbitrary  continuous (left) action of $\SL(2,\R)$ on a locally compact space $X$.
Let $\nu$ be any of the following type of Borel probability measures on $X$:

\begin{enumerate}
\item a horocycle probability invariant measure, that is, a Borel probability measure invariant under the action of the unipotent 
subgroup $h_\R$ on $X$;
\item a (normalized) horocycle arc, that is, a measure of the form
$$
\frac{1}{S} \int_0^S   (h_s)_* (\delta_x) ds \,, \quad    \text{ for some } (x,S) \in X \times \R^+;
$$
\item a (normalized) circle arc, that is, a measure of the form
$$
 \frac{1}{\Theta} \int_0^\Theta  ( r_\theta)_* (\delta_x) d\theta   \,, \quad    \text{ for some } (x,\Theta) \in X \times \R^+.
$$
\end{enumerate}
Our first result, on push-forwards of horocycle invariant measures, and of horocycle and circle arcs,  can be stated as follows:

\begin{theorem} \label{thm:HP} Let $\nu$ be any Borel probability measure in the above list. If the weak* limit
$$
\mu:=\lim_{T\to +\infty}  \frac{1}{T}\int_0^T   (g_t)_* (\nu) dt 
$$
exists and is $h_\R$- ergodic, then there exists a set $Z\subset \R$ of zero upper density
such that in the weak* topology
$$
\lim_{t\not \in Z}    (g_t)_* (\nu) =\mu\,.
$$ 
\end{theorem}

Our second result, on Birkhoff genericity, is as follows.

Let $I\subset \T$ be a compact subinterval.   Let $\delta^I_\infty$ denote the Dirac mass at the point at infinity of the one-point compactification $S_I$ of the 
semi-infinite cylinder $[1, +\infty) \times I$.  

\begin{theorem} \label{thm:BG}  For any sequence $(\pi_n)$ of probability measures converging to the Dirac measure $\delta^I_\infty$, the following holds. 
Let us assume that in the weak* sense
$$
\mu:=\lim_{n\to+\infty} \int_{S_I} \frac{1}{T}  \int_0^T  (g_t\circ r_\theta)_* (\delta_x)dt   d\pi_n (T, \theta)
$$
then there exists a set ${\mathcal Z }\subset [1, +\infty) \times I$ with $\lim_{n\to +\infty}\pi_n ({\mathcal Z}) =0$
such that  in the weak* topology we have
$$
\lim_{(T,\theta)\not\in {\mathcal Z}}   \frac{1}{T} \int_0^T  (g_t\circ r_\theta)_* (\delta_x)dt = \mu\,.
$$
\end{theorem} 

As a consequence of Theorem~\ref{thm:HP} we can derive the following:

\begin{corollary}  
\label{cor:horo_limit}
Let $\nu$ be a horocycle invariant measure, or the probability measure uniformly distributed on
a horocycle or circle arc on the moduli space. Let $\mu$ be the unique affine probability $\SL(2, \R)$-invariant
measure supported on the affine sub-orbifold $ \overline{ \SL(2, \R)(\text{\rm supp}(\nu))}$. 
There exists a set $Z \subset \R$ of zero upper density such that
$$
\lim_{t\not \in Z}    (g_t)_* (\nu) =\mu\,.
$$
\end{corollary}

\begin{conjecture} Corollary~\ref{cor:horo_limit} holds with exceptional set $Z=\emptyset$.
\end{conjecture}
As mentioned above this conjecture has been proven by M.~Bainbridge, J.~Smillie and B.~Weiss  for certain invariant orbifolds in the stratum $H(1,1)$ 
of Abelian differentials with two simple zeros on genus $2$ surfaces (see \cite{BSW}, Theorems  1.5 and 12.7). 
J.~Chaika, J.~Smillie and B.~Weiss have recently announced that the Teichm\"uller horocycle flow (in genus $2$) has orbits which are not contained in the support of their limit measures and orbits which have no (unique) limit measure. These results however do not  contradict our conjecture. 

\smallskip
By the argument of \cite{EMa}) we can also derive the following improved version
of the ``weak asymptotic formula'' for the counting function of cylinders in translation surfaces and rational billiards
(compare Theorem 1.7 in \cite{EM} or Theorem 2.12 in \cite{EMM}).
Let $Q$ be a rational polygon, and let $N(Q,T)$ denote the number of cylinders of periodic trajectories of length at most $T>0$ 
for the billiard flow on $Q$.

\begin{corollary} There exists a constant $C_Q$ (a Siegel-Veech constant) and a set $Z_Q\subset \R$ of zero upper density such that
$$
\lim_{t \not \in Z_Q}   \frac{N(Q, e^t)}{ e^{2t}} =  C_Q\,.
$$
\end{corollary} 

For {\it horospherical measures} we can prove a stronger result.  By a horospherical measure we mean any measure supported on a leaf of the strong stable foliation 
of the Teichm\"uller geodesic flow,  absolutely continuous with continuos density with respect  to the canonical affine measure on the leaf, and with conditional measures along
horocycle orbits equal to one-dimensional Lebesgue measures (see Section~\ref{sec:horospheres}). 

\begin{theorem}
\label{thm:horospheres}
Let $\nu$ be any horospherical probability measure supported on the stable leaf at $x\in \mathcal H_g$ and let $\mu$ denote the
unique $\SL(2, \R)$-invariant affine ergodic probability measure supported on $\overline{\SL(2, \R) x}$. In the weak* topology, we have
$$
\lim_{t\to +\infty} (g_t)_* (\nu)   = \mu\,.
$$
\end{theorem}

Theorem \ref{thm:BG} has the corollaries stated below.

\begin{corollary} \label{cor:BG1}  Let $x\in X$ and let $I \subset \T$. If the weak* limit
$$
\mu:=\lim_{T\to +\infty} \frac{1}{T}  \int_0^T  \frac{1}{\vert I\vert} \int_I   (g_t\circ r_\theta)_* (\delta_x)d\theta dt 
$$
exists and is $g_\R$-ergodic, then for Lebesgue almost all $\theta \in I$  there exists a set $Z_\theta \subset \R$ of zero lower density
such that   we have the weak* limit
$$
\mu:=\lim_{T\not\in Z_\theta}  \frac{1}{T} \int_0^T (g_t\circ r_\theta)_* (\delta_x) dt \,.
$$
\end{corollary}

Corollary~\ref{cor:BG1} has been recently proven by O.~Khalil (see~\cite{Kha}, Th. 1.1) in greater generality by a different, direct argument. Khalil's argument is
based on an ``adaptation of the weak-type maximal inequality  and follows similar lines to the proof of the classical Birkhoff ergodic theorem''. We believe that our indirect argument can be generalized to yield a result identical to that of Khalil.

\smallskip
In the motivating case when $X={\mathcal H}_g$ is the moduli space of Abelian differentials on Riemann surfaces of genus
$g\geq 2$, the results of Eskin, Mirzakhani and Mohammadi, \cite{EM} and \cite{EMM}, prove that for every
$x \in {\mathcal H}_g$ there exists a unique probability $\SL(2, \R)$-measure, absolutely continuous on the
affine orbifold $\overline {SL(2,\R)x}$, such that the hypotheses of Theorem~\ref{thm:HP} and Corollary
\ref{cor:BG1} hold (see \cite{EMM}, Theorems 2.6 and 2.10).  
In this case the Birkhoff genericity in almost all directions, which corresponds to the statement of  Corollary~\ref{cor:BG1} with exceptional sets
$Z_\theta= \emptyset$, for almost all $\theta\in \T$,  was proved by Chaika and Eskin (see \cite{CE}, Theorem 1.1) by a different methods based 
on ideas and results  from~\cite{EM} and~\cite{EMM}. 

\smallskip
By our method, we can establish Birkhoff genericity in almost all directions in an abstract setting only under a
stronger hypothesis.

\begin{corollary} \label{cor:BG2}  Let $x\in X$ and  let $I \subset \T$.  If the weak* limit
$$
\mu:=\lim_{\pi \to \delta^I_{\infty}} \int_{S_I} \frac{1}{T}  \int_0^T  (g_t\circ r_\theta)_* (\delta_x)dt   d\pi (T, \theta)
$$
exists as $\pi$ varies over all compactly supported, absolutely continuous probability measure on $S_I$ with smooth density,  
and if $\mu$ is $g_\R$-ergodic, then  for almost all $\theta \in \T$ we have the weak* limit
$$
\mu:=\lim_{T\to+\infty}  \frac{1}{T} \int_0^T (g_t\circ r_\theta)_* (\delta_x) dt  \,.
$$
\end{corollary}
\medskip

Our results on the {\it Oseledets theorem} in the generic direction are formulated below. We consider an action the group $\SL(2, \R)$ on a continuous vector bundle in the setting of the paper by C.~Bonatti, A. Eskin and A.~Wilkinson \cite{BEW}.  We recall briefly the main hypothesis of their work.  Let $H \to X$ be a continuous vector bundle over a separable metric space $X$  with fiber a finite dimensional vector space.  
Suppose that $\SL(2,\R)$ acts on $H$ by linear automorphisms on the fibers and a given  action on the base which preserves a probability measure $\mu$ on $X$. 
Assume that $H$ is equipped with a Finsler structure (that is, a continuous choice of norm $\vert \cdot \vert_x$ on the fibers of $H$). For any $g\in \SL(2, \R)$ let 
$$
\Vert  g  \Vert_x =  \sup_{ \rm v\in H_x\setminus\{0\}}   \frac  {\vert  g(\rm v) \vert_{g(x)} }{ \vert \rm v \vert_x } \,.
$$
The cocycle is called {\it uniformly Lipschitz}  with respect to the Finsler structure,  if there exists a constant $K>0$ such that for $(x,t) \in X\times \R$, 
$$
\log \Vert  g_t \Vert_x    \leq   K t   \,.
$$
We remark that all uniformly Lipschitz cocycles trivially satisfy the integrability condition of \cite{BEW}:
$$
\int_X  \sup_{t\in [-1, 1]} \log \Vert  g_t \Vert_x \, d\mu(x)     \, < \,+ \infty.
$$
The cocycle is called \emph{irreducible}  with respect to the $\SL(2, \R)$-invariant measure $\mu$ on $X$ if it does not admit non-trivial  $\mu$-measurable 
$\SL(2,\R)$-invariant sub-bundles.

Let $I\subset \T$ be a compact subinterval.   Let $\delta^I_\infty$ denote the Dirac mass at the point at infinity of the one-point compactification $S_I$ of the 
semi-infinite cylinder $[1, +\infty) \times I$.  

\begin{theorem} 
\label{thm:OG}
Assume that the $\SL(2,\R)$ cocycle on $H$ is uniformly Lipschitz and irreducible with respect to the $\SL(2, \R)$-invariant probability ergodic measure $\mu$ on $X$. Let $\lambda_\mu$ denote
the top Lyapunov exponent of the diagonal cocycle $g^H_\R$ on $H$ with respect to the measure $\mu$ on $X$.  
For any sequence $(\pi_n)$ of probability measures converging to the Dirac measure $\delta^I_\infty$, the following holds. 
Let $x\in X$ be any point such that (in the weak* topology) we have 
$$
\mu:=\lim_{n\to+\infty} \int_{S_I} \frac{1}{T}  \int_0^T  (g_t\circ r_\theta)_* (\delta_x)dt   d\pi_n (T, \theta)\,.
$$
It follows that  there exists a set ${\mathcal Z }\subset [1, +\infty) \times I$ with $\lim_{n\to +\infty}\pi_n ({\mathcal Z}) =0$
such that, for all ${\rm v} \in H_x\setminus \{0\}$, we have
$$
\lim_{(t,\theta) \not\in {\mathcal Z}}  \frac{1}{t}  \log  \vert  g^H_t (r_\theta(\rm v) ) \vert_{g_t (r_\theta(x))}  = \lambda_\mu.
$$
\end{theorem}

Theorem \ref{thm:OG} has the corollaries stated below.

\begin{corollary} \label{cor:OG1} 
Assume that the $\SL(2,\R)$ cocycle on $H$ is uniformly Lipschitz and irreducible with respect to the $\SL(2, \R)$-invariant probability ergodic measure $\mu$ on $X$. Let $\lambda_\mu$ denote the top Lyapunov exponent of the diagonal cocycle $g^H_\R$ on $H$ with respect to the measure $\mu$ on $X$.  Let $x\in X$ 
be any point such that (in the weak* topology) we have 
$$
\mu:=\lim_{T\to +\infty} \frac{1}{T}  \int_0^T  \frac{1}{\vert I\vert} \int_I   (g_t\circ r_\theta)_* (\delta_x)d\theta dt 
$$
It follows that, for Lebesgue almost all $\theta \in I$,  there exists a set $Z_\theta \subset \R$  of zero lower density such that,
for all ${\rm v} \in H_x\setminus \{0\}$, we have
$$
\lim_{t \not\in Z_\theta}  \frac{1}{t}  \log  \vert  g^H_t (r_\theta(\rm v) ) \vert_{g_t (r_\theta(x))}  = \lambda_\mu.
$$
\end{corollary}
In the motivating case when $X={\mathcal H}_g$ is the moduli space of Abelian differentials on Riemann surfaces of genus
$g\geq 2$,  Oseldets genericity in almost all directions, which corresponds for the top exponent to the statement of 
Corollary~\ref{cor:OG1} with exceptional sets $Z_\theta= \emptyset$,  for almost all $\theta\in \T$, was also proved by Chaika and Eskin 
(see \cite{CE}, Theorems 1.2 and 1.5)  by an argument based on~\cite{EM},~\cite{EMM}. 

\smallskip
By our method, we can establish Oseledets regularity in almost all directions in an abstract setting only under a
stronger hypothesis.
\begin{corollary} \label{cor:OG2}  Let $x\in X$ and  let $I \subset \T$.  If the weak* limit
$$
\mu:=\lim_{\pi \to \delta^I_{\infty}} \int_{S_I} \frac{1}{T}  \int_0^T  (g_t\circ r_\theta)_* (\delta_x)dt   d\pi (T, \theta)
$$
exists, over all compactly supported, absolutely continuous probability measure $\pi$ on $S_I$ with smooth density,  
and is $g_\R$-ergodic, then  for almost all $\theta \in \T$ and for all ${\rm v} \in H_x\setminus \{0\}$, we have
$$
\lim_{t \to +\infty}  \frac{1}{t}  \log  \vert  g^H_t (r_\theta(\rm v) ) \vert_{g_t (r_\theta(x))}  = \lambda_\mu.
$$
\end{corollary}

Whenever Theorem \ref{thm:OG} and its corollaries can be applied to all exterior products of the cocycle, it  implies results for all Lyapunov exponents.
In the special case of
the action of $\SL(2, \R)$ on the moduli space of Abelian differentials the theorem can indeed be applied (as in the paper by Chaika and Eskin \cite{CE}) to all exterior powers 
of the Kontsevich--Zorich cocycle by  reduction to the irreducible component thanks to the semi-simplicity theorem of  S.~Filip \cite{Fi}.  The uniform Lipschitz property of the
Kontsevich--Zorich cocycle with respect to the Hodge norm was proved by the author in \cite{Fo} (see also \cite{FMZ}).

\section*{Acknowledgements}

We are very grateful to Barak Weiss and Ronggang Shi for pointing out an error in the first draft of this paper. As a consequence our results on
Birkhoff genericity and Oseledets regularity have been significantly weakened. We are grateful to  D.~Aulicino, A.~Brown,  A.~Eskin, Carlos Matheus and 
R.~Trevi\~no for several discussions about the topics of this paper. This research was supported by the
NSF grant DMS 1600687.

\section{Limits of Geodesic Push-Forwards of Horocycle measures}
\label{sec:LGP} 
  
In this section we prove Theorem~\ref{thm:HP}. 
  
\begin{proof}

We begin explaining the argument in the case of push-forwards of horocycle invariant probability measures. The
other cases can be treated similarly, and require some additional considerations.

\smallskip
Let $B_1$ be the set of all Borel measures of total mass at most one on the one-point compactification $\hat X$ of the locally compact space $X$ and  let $\mathcal N: \R \to B_1$ be the map defined as $\mathcal N(t) = (g_t)_* (\nu)$ for all $t\in \R$. The range of the map $\mathcal N$ is contained in the closed subspace of probability measures, invariant under the (unstable) horocycle 
flow, since we are assuming that $\nu$ is invariant under the (unstable) horocycle flow. 

The space $B_1$, endowed with the topology of the weak* convergence, is a metrizable (separable) compact space.  The map 
$\mathcal N: \R^+ \to B_1$ is clearly continuous.  In fact, every continuous functions on $\hat X$ is bounded, and all measures 
$ (g_t)_* (\nu)$ are probability measures. 
    
 \smallskip
 Let $u_T$ denote the uniform measure on $[0,T]$. The push-forward $\mathcal N_* (u_T)$ defines a Borel probability measure
on the compact metric space $B_1$, that is, for any  continuous function $F$ on the space $B_1$  
(with respect to the weak* topology on $B_1$), we have
$$
 \mathcal N_* (u_T)(F) :=  \frac{1}{T}  \int_0^T   F((g_t)_*(\nu)) dt\,.
$$
Let $U$ denote any weak* limit of $\mathcal N_* (u_T)$ as $T\to +\infty$ in the space of probability measures  on the compact metric
space $B_1$.
We claim that $U$ is a delta mass $\delta_\mu$ at the probability measure $\mu \in B_1$. We remark  that, as $\nu$ is horocycle invariant, since the subset of horocycle invariant probability measures is closed  with respect to the weak* topology, the measure 
$U$ is supported there.  

For any continuous (compactly supported) function $f$ on the locally compact space $X$, the function $F_f$ defined as
$$
F_f (\nu) = \nu (f) \,,  \quad \text{ \rm for all } \nu \in B_1\,,
$$
is continuous with respect to the weak* topology on $B_1$ (by definition of the weak* topology). 
By our assumptions, for all functions $F_f$ we have that
$$
U(F_f) = \mu(f)\,.
$$
In fact, for every continuous function $f$ with compact support on the space $X$ 
we have that
$$
\mathcal N_* (u_T)(F_f) :=  \frac{1}{T}  \int_0^T  [(g_t)_*(\nu)] (f) dt    \to   \mu (f) \,.
$$
Since $\mu$ is ergodic with the respect to the horocycle flow and $U$ is supported on the subset of horocycle invariant 
measures, from the identity
\begin{equation}
\label{eq:convex1}
U(F_f) = \int  F_f(\nu)   dU(\nu)  =  \int  \nu(f)   dU(\nu)   = \mu(f)\,
\end{equation}
 we derive that $\nu=\mu$ for $U$-almost all $\nu \in B_1 $. This in turn implies that the probability measure $U$, as
a probability measure essentially supported on the singleton  $\{\mu\} \subset B_1$, is equal to a Dirac mass $\delta_\mu$. 

It follows that, for every open neighborhood $\mathcal V$ of $\mu$ in the (metric) space $B_1$, we have
\begin{equation}
\label{eq:zero_freq}
\limsup_{T\to +\infty}  \frac{1}{T} \text{ Leb} \{ t \in [0,T] :   (g_t)_*(\nu) \notin \mathcal V\}  =0 \,.
\end{equation}
In fact, let us assume the above statement does not hold. It follows that there exists an open neighborhood $\mathcal V$ of
$\mu$ in $B_1$ and a diverging sequence $\{T_n\}$ such that
$$
\lim_{n\to +\infty}  \frac{1}{T_n} \text{ Leb} \{ t \in [0,T_n] :   (g_t)_*(\nu) \notin \mathcal V\}  =c >0\,.
$$
There exists a continuous non-negative function $F_{\mathcal V}$ on the compact metrizable space $B_1$ such that $F_{\mathcal V}\equiv 1$ on the complement of $\mathcal V$ (a closed set) and $F_{\mathcal V}(\mu)=0$. It follows that 
$$
\frac{1}{T_n} \int_0^{T_n}  F_{\mathcal V} ((g_t)_*(\nu)) dt    \geq    c >0  \,,
$$
hence, for every weak limit $U$ of the sequence of measures $\mathcal N_* (u_{T_n})$,  we have $U(F_\mathcal V) \geq c>0$
while $\delta_\mu (F_{\mathcal V}) =0$, which is a contradiction. 

We conclude that there exists  $Z \subset \R^+$ of zero upper density  such that
\begin{equation}
\label{eq:limit}
\lim_{ t\not \in Z} (g_t)_*(\nu) = \mu \quad \hbox {in the weak* topology}.
\end{equation}
Let $\{\mathcal V_n\}$ be a basis of neighborhoods of the measure $\mu$ in $B_1$.  By formula
\eqref{eq:zero_freq} for every sequence $\{\epsilon_n\}$ of positive numbers, converging to zero, there exists a diverging, 
increasing sequence $\{T_n\}$ such that  
$$
  \sup_{T\geq T_n} \frac{1}{T} \text{ Leb} \{ t \in [0,T] :   (g_t)_*(\nu) \not\in \mathcal V_n\}   \leq  \epsilon_n \,. 
$$
Let $Z$ be the set defined as follows:
$$
Z:=  \cup_{n\in \N}   \{ t \in [T_n,T_{n+1}] :   (g_t)_*(\nu) \not\in \mathcal V_n\} \,.
$$
Let us find under what conditions $Z$ has zero upper density. For any $T>0$ sufficiently large there exists
$n\in \N$ such that $T \in [T_n, T_{n+1}]$. We have
$$
\begin{aligned}
\text{Leb} (Z \cap [0,T]) &\leq  \sum_{k=1}^{n-1} \text{Leb} (Z \cap [T_k, T_{k+1}])  +   \text{Leb} (Z \cap [T_n, T]) 
\\ &\leq   \sum_{k=1}^{n-1} \epsilon_k T_{k+1} +  \epsilon_n T \,.
\end{aligned}
$$
It is therefore enough to choose the sequences recursively so that
$$
\frac{1}{T_n} \sum_{k=1}^{n-1} \epsilon_k T_{k+1}   + \epsilon_n \to   0\,.
$$
It is clear by the definition of the set $Z\subset \R$  that formula \eqref{eq:limit} holds.

\bigskip
For the cases when $\nu$ is a probability measure supported on a horocycle arc or an arc of circle, the argument is
similar but we have to prove that any weak* limit $U$ of $\mathcal N_* (u_T)$ as $T\to +\infty$ is supported on the
subspace of horocycle invariant measures.

 Let $U$ a weak* limit of the measures $\mathcal N_* (u_T)$  with support not contained in the subspace of the horocycle
 invariant measures. There exists a measure $\nu_0 \in \text{supp}(U)$ which is not invariant under the (unstable) horocycle
 flow, hence here exist a function $f_0\in C^0_0 (X)$  and a real number $\tau\not =0$ such that 
$$
\int f_0\circ h_\tau   \, d\nu_0   \not = \int f_0  \, d\nu_0   \,.
$$
Since $\nu_0\in \text{supp}(U)$ and since $B_1$ is a locally convex metrizable space, there exists a closed convex neighborhood $\mathcal C \subset B_1$ such 
that $U(\mathcal C) >0$ and $U (\partial {\mathcal C}) =0$, and by continuity 
$$
\int f_0\circ h_\tau   \, d\nu    \not = \int f_0  \, d\nu_0 \,,  \quad \text{ for all } \nu \in \mathcal C\,.
$$
Let  $\nu_{\mathcal C}$ denote the measure
$$
\nu_{\mathcal C} =  \frac{1}{ U(\mathcal C)}\int_{\mathcal C}  \nu dU(\nu) \, \in \,  {\mathcal C} \,.
$$
We claim that $\nu_{\mathcal C}$ is invariant under the horocycle flow, a contradiction. In fact, since $U$ is a weak*
limit of the family $\{\mathcal N_*(u_T)\}$ there exists a diverging sequence $(T_n)$ such that 
$\mathcal N_*(u_{T_n}) \to U$, and since by assumption that $U (\partial C) =0$, we have
$$
U(\mathcal C) \nu_{\mathcal C} = \lim_{n\to +\infty}    \int_{\mathcal C}  \nu d \mathcal N_*(u_{T_n}) (\nu) \,.
$$
Now, by construction, for all $f\in C^0_0(X)$,  we have 
$$
 {[}\int_{\mathcal C}  \nu d \mathcal N_*(u_{T_n}) (\nu){]}(f\circ h_\tau -f) =\frac{1}{T_n}  \int_0^{T_n} \chi_C ( (g_t)_\ast \nu)  
 \nu (f\circ h_\tau -f)   dt \,.
$$
It is therefore enough to prove that the RHS in the above formula converges to zero. When $\nu$ is the uniformly
distributed measure on a horocycle arc $\{ h_s (x) \vert s\in [0, S]\}$, we have, uniformly with respect to $t\in \R$,
$$
 \begin{aligned} 
 \frac{1}{S}  &\int_0^S (f\circ h_\tau -f) (g_t h_s x)ds = 
  \frac{1}{S} \int_0^S (f\circ h_\tau -f) (  h_{e^{2t}s} \circ g_t x)ds 
  \\& =  \frac{1}{S}  [\int_S^{S+ e^{-2t}\tau}  f(h_{e^{2t}s} \circ g_t x) ds -   \int_0^{e^{-2t}\tau}  f(h_{e^{2t}s} \circ g_t x) ds]
  \to 0\,.
  \end{aligned} 
$$
When $\nu$ is the uniformly distributed measure on a circle arc $\{ r_\theta (x) \vert \theta\in [0, \Theta]\}$, it is
a standard argument that the push-forward of a circle arc can be well approximated by a union of horocycle arcs. We include the argument for completeness. 

There exists a constant $C>0$ such that, for all $\theta\in [-\pi/4, \pi/4]$ and for all $t\in \R$, we have
$$
\text{\rm dist} ( g_t \circ r_\theta,  g_t  \circ h_\theta) \leq    C ( e^t +e^{-t}) \theta^2 \,.
$$
As a consequence, for any $\alpha \in (1/2, 2)$ we can approximate the push forward of a circle arc $\{r_\theta x \vert \theta \in [0, \Theta]\}$ by a union of  at most $\Theta e^{\alpha t}$  push-forwards of horocycle arcs of length $\ell_t \in [e^{-\alpha t}, 2e^{-\alpha t}]$. 
By the above estimate  the error in computing the integral of a continuos function $f$, of unit Lipschitz constant with respect to the
$\SL(2,\R)$ action, will be of size
$$
4\Theta C ( e^t +e^{-t}) e^{-2\alpha t}  \leq   8 C \Theta e^{(1-2\alpha)t}  \to 0 \,.
$$
The claim is thus reduced to prove that for any family of intervals $\{[a_t, b_t] \}$ we have
$$
\lim_{t\to \infty} \frac{1}{\ell_t } \int_{a_t}^{b_t}   (f\circ h_\tau -f) (g_t h_s x  ) ds  =0 \quad \text{ uniformly on } x\in X\,.
$$
The proof of the above limit is straightforward since 
$$
\begin{aligned}
\vert   \frac{1}{\ell_t } \int_{I_t}  & (f\circ h_\tau -f) (g_t h_s x  ) ds \vert  \\ &=   \frac{1}{\ell_t }  \vert \int_{b_t}^{b_t+ e^{-2t}\tau}  f(h_{e^{2t}s} \circ g_t x)ds  -    \int_{a_t}^{a_t+ e^{-2t}\tau} f(h_{e^{2t}s} \circ g_t x) ds  \vert \,,
\end{aligned}
$$
hence, as $\ell_t \geq e^{-\alpha t}$ with $\alpha <2$, it follows that  for any  $f\in C^0_0(X)$, the above averages converge to zero, uniformly with respect to $x \in X$. 

We have thus completed the proof of the claim that $\nu_{\mathcal C}$ is in all cases invariant under the horocycle flow, and since 
$\nu_{\mathcal C} \in \mathcal C$ we have reached a contradiction. 

\end{proof}

\section{Birkhoff genericity in almost all directions}
\label{sec:BG}

In this section we prove Theorem~\ref{thm:BG} and Corollaries~\ref{cor:BG1} and~\ref{cor:BG2}.

\begin{proof}[Proof of Theorem~\ref{thm:BG}]
Given $x\in X$, let us consider the map $G: [1, +\infty) \times I \to B_1$  the  to the space of measures  on $\hat X$ of total mass at most $1$, 
given for every $T \geq 1$, for every  $\theta \in I\subset \T$, and for every $f \in C^0(\hat X)$ by the formula
$$
G (T,\theta) (f) :=  \frac{1}{T}  \int_0^T   f (g_t r_\theta x) dt \,.
$$
For any compact interval $I \subset \T$, let us consider the family of push-forwards
$$
\{G_* ( \pi),  \text{ for all probability measure } \pi \text{ on } [1, +\infty) \times I \}\,.
$$
Let $\Pi$ be any weak* limit of the above family in the following sense.  There exists a sequence $(\pi_n)$ 
which converges in the weak* topology to the delta mass at the one-point compactification $S_I$ of the  cylinder
$[1, +\infty) \times I$ such that
$$
 G_* ( \pi_n)  \to \Pi  
$$
in the weak* topology on the space of measures on $B_1$.

We claim that $\Pi$ is a Dirac mass supported at $\mu$. By our hypotheses for all functions $F_f$ we have
$$
\Pi(F_f) = \mu (f) \,.
$$
In fact, by hypothesis we have
$$
\begin{aligned}
G_* ( \pi_n )(F_f) &= \int_1^{+\infty} \int_I     F_f( G(T,\theta)) d\pi_n(T, \theta) \\&=
\int_1^{+\infty}  \int_I     \frac{1}{T}  \int_0^T   f (g_t r_\theta x) dt  d\pi_n(T, \theta) \to   \mu (f) \,.
\end{aligned}
$$
It the follows that, for all $f \in C^0(X)$,
\begin{equation}
\label{eq:convex2}
\mu(f) = \Pi(F_f) =  \int \nu (f) d\Pi(\nu) \,.
\end{equation}
We claim that the support of $\Pi$ is contained in the closed subspace of $B_1$ of probability measures
invariant under the geodesic flow. Let us assume that it is not the case. It follows that
there exists a measure $\nu_0 \in \text{supp}(\Pi)$ which is not invariant under the geodesic
flow.  There exist a function $f_0\in C^0_0 (X)$  and a real number $\tau\not =0$ such that 
$$
\int f_0\circ g_\tau   \, d\nu_0   \not = \int f_0  \, d\nu_0   \,.
$$
Since $\nu_0\in \text{supp}(\Pi)$ and since $B_1$ is a locally convex metrizable space, there exists a closed convex neighborhood $\mathcal C \subset B_1$ 
such that $\Pi(\mathcal C) >0$ and $\Pi(\partial \mathcal C) =0$, and by continuity 
$$
\int f_0\circ g_\tau   \, d\nu    \not = \int f_0  \, d\nu_0 \,,  \quad \text{ for all } \nu \in \mathcal C\,.
$$
Let  $\nu_{\mathcal C}$ denote the measure
$$
\nu_{\mathcal C} =  \frac{1}{ \Pi(\mathcal C)}\int_{\mathcal C}  \nu d\Pi(\nu) \, \in \,  {\mathcal C} \,.
$$
We claim that $\nu_{\mathcal C}$ is invariant under the geodesic flow, a contradiction. In fact, since $\Pi$ is a weak*
limit of the family $\{G_*(\pi_n)\}$ and  $\Pi(\partial \mathcal C) =0$, we have
$$
\Pi(\mathcal C) \nu_{\mathcal C} = \lim_{n\to +\infty}    \int_{\mathcal C}  \nu dG_*(\pi_n ) (\nu) \,.
$$
Now, by construction, for all $f\in C^0_0(X)$,  we have 
$$
\begin{aligned}
 {[}\int_{\mathcal C}  &\nu dG_*(\pi_n) (\nu){]}(f\circ g_\tau -f) \\
 &=  \frac{1}{\vert I\vert} \int_{G^{-1}(\mathcal C)} \frac{1}{T}
(\int_0^T (f\circ g_\tau -f) (g_t r_\theta x) dt) d\pi_n(T,\theta) \\ &=   \frac{1}{\vert I \vert}  \int_{G^{-1}(\mathcal C)}  
\frac{1}{T} (\int_{T}^{T+\tau} f (g_t r_\theta x) dt  - \int_0^\tau f (g_t r_\theta x) dt)  d\pi_n(T,\theta)   \to 0   \,.
\end{aligned}
$$
The claim that $\nu_{\mathcal C}$ is invariant under the geodesic flow follows, and since $\nu_{\mathcal C} \in
\mathcal C$ we reached a contradiction. It follows that $\Pi$ is supported on the subspace of $g_{\R}$-invariant
measures. 

\smallskip
From formula \eqref{eq:convex2} and from the ergodicity of the measure $\mu$ with respect to the geodesic 
flow, it follows that $\Pi= \delta_\mu$ is a Dirac mass at $\mu$. We have thus proved that 
$$
\lim_{n\to +\infty}   G_* (\pi_n) =   \delta_\mu \,.
$$
From the above conclusion we immediately derive that, 
for every neighborhood $\mathcal V$ of $\mu$ in the (metric) space $B_1$,  we have
\begin{equation}
\label{eq:measure_bound_1}
\lim_{n\to +\infty}   \pi_n (\{ (T,\theta)\in [1, +\infty) \times I \vert  G(T,\theta) \not\in \mathcal V\} ) =0\,.
\end{equation}
Since $\pi_n (\{\infty\})=0$, it follows that for every $\epsilon>0$ there exists $T_\epsilon>1$ such that
$$
\pi_n (\{ (T,\theta)\in [T_\epsilon, +\infty) \times I \vert  G(T,\theta) \not\in \mathcal V\} ) < \epsilon, \quad \text{ for all }n \in\N\,.
$$
Let $(\mathcal V_k)$ be a basis of neighborhoods of $\mu$ in $B_1$ and let $(\epsilon_k)$ be any {\it summable} sequence of positive real numbers. 
There exists an increasing diverging sequence $(T_k)$ such that 
$$
\pi_n (\{ (T,\theta)\in [T_k, +\infty) \times I \vert  G(T,\theta) \not\in \mathcal V_k\} ) < \epsilon_k, \quad \text{ for all }n \in\N\,.
$$
Let then $\mathcal Z$ be the set such that
$$
{\mathcal Z} \cap [T_k, T_{k+1}) = \{ (T,\theta)\in [T_k, T_{k+1}) \times I \vert  G(T,\theta) \not\in \mathcal V_k\}\,.
$$
It is clear that by construction we have
$$
\lim_{(T,\theta)\not\in {\mathcal Z}}   G(T,\theta) = \mu\,.
$$
Finally, since $\pi_n \to \delta^I_\infty$ it follows that, for any $k\in \N$, we have
$$
\lim_{n\to+\infty} \pi_n ( [1, T_k) \times I )=0\,, 
$$
while by construction, for all $n\in \N$, we have
$$
\lim_{k\to+\infty} \pi_n \left( {\mathcal Z} \cap ([T_k, +\infty) \times I)\right) \leq \lim_{k\to+\infty}  \sum_{j\geq k} \epsilon_j =0 \,.
$$
The statement of the theorem follows.
\end{proof}

\begin{proof}[Proof of Corollary \ref{cor:BG1}]
Let $(\tau_n)$ be any sequence of probability measures on $[1, +\infty)$ which converges to the Dirac mass at the point at infinity.
Let  $\Theta_I$ denote the normalized Lebesgue measure on the interval $I\subset \T$ and let then $(\pi_n)$ be the sequence of probability measures 
on $[1, +\infty) \times I$ defines as
$$
\pi_n :=   \tau_n \times  \Theta_I  \,, \quad \text{ for all } n\in \N\,.
$$
By the hypothesis of the corollary that
$$
\mu:=\lim_{T\to +\infty} \frac{1}{T}  \int_0^T  \frac{1}{\vert I\vert} \int_I   (g_t\circ r_\theta)_* (\delta_x)d\theta dt 
$$
it follows that the hypothesis of Theorem~\ref{thm:BG} holds for the sequence $(\pi_n)$. Therefore there exists a
set $\mathcal Z$ such that   $\lim_{n\to +\infty}\pi_n ({\mathcal Z}) =0$
such that  in the weak* topology we have
$$
\lim_{(T,\theta)\not\in {\mathcal Z}}    G(T, \theta) = \lim_{(T,\theta)\not\in {\mathcal Z}}   \frac{1}{T} \int_0^T  (g_t\circ r_\theta)_* (\delta_x)dt = \mu\,.
$$

From the above conclusion we derive that, 
for every sequence of probability measure $(\tau_n)$ weakly converging to the delta mass at $+\infty\in [1, +\infty]$ and
for every neighborhood $\mathcal V$ of $\mu$ in the (metric) space $B_1$,  we have
\begin{equation}
\label{eq:Bzero_lim_measure}
\lim_{n\to +\infty}   (\tau_n \times\Theta_I)(\{ (T,\theta)\in [1, +\infty) \times I \vert  G(T,\theta) \not\in \mathcal V\} ) =0\,.
\end{equation}
We claim that for every $\mathcal V$, there exists a full measure set $\mathcal F_{\mathcal V} \subset I$ such that for all $\theta \in \mathcal F_{\mathcal V}$ we have
$$
\liminf_{\mathcal T \to +\infty}  \frac{1}{\mathcal T}  \text{Leb}( \{T\in [1, \mathcal T] \vert   G(T, \theta)\not \in \mathcal V\} )=0\,.
$$
Otherwise there exists a positive measure set $\mathcal P_\mathcal V \subset I$ such that for all $\theta\in \mathcal P_\mathcal V $
$$
\liminf_{\mathcal T \to +\infty}  \frac{1}{\mathcal T}  \text{Leb}( \{T\in [1, \mathcal T] \vert   G(T, \theta)\not \in \mathcal V\}) >0\,.
$$
By the Egorov theorem it follows that there exists a  sequence of times $(\mathcal T_n)$ and a positive measure subset 
$\mathcal P'_\mathcal V  \subset \mathcal P_\mathcal V$ such that  the sequence
$$
\inf_{{\mathcal T}\geq {\mathcal T}_n} \frac{1}{\mathcal T}  \text{Leb}( \{T\in [1, \mathcal T] \vert   G(T, \theta)\not \in \mathcal V\})
$$
converges uniformly to a continuous positive function on $\mathcal P'_\mathcal V$. It is then possible to construct a sequence $(\tau_n)$ of 
probability measures on $[1, +\infty)$, weakly converging to the delta mass at $+\infty\in [1, +\infty]$, which contradicts the conclusion in
formula \eqref{eq:Bzero_lim_measure}. Hence the above claim is proved.

Let $(\mathcal V_n)$ be a basis of neighborhoods of $\mu$ in $B_1$ and let $\mathcal F_I$ denote the full measure set defined as 
$$
\mathcal F_I := \bigcap_{n\in \N} \mathcal F_{\mathcal V_n}\,.
$$
From the above claim it follows that for all $\theta\in \mathcal F_I$, and for all $n\in \N$, we have 
$$
\liminf_{\mathcal T \to +\infty}  \frac{1}{\mathcal T}  \text{Leb}( \{T\in [1, \mathcal T] \vert   G(T, \theta)\not \in \mathcal V_n\} )=0\,.
$$
In particular, for any sequence $(\epsilon_n)$ of positive real numbers, converging to zero, we have that there exists an increasing diverging sequence
$({\mathcal T}_n) \subset [1, +\infty)$ such that 
$$
\frac{1} { {\mathcal T}_n}  \text{ Leb} \left(\{ T \in [1, {\mathcal T}_n]  \vert  G(T,\theta) \not\in \mathcal V_n \}\right) < \epsilon_n\,.
$$
Such sequence can be constructed recursively as follows. For any finite increasing sequence $\{\mathcal T_k\vert k\ \leq n\}$, and for any $\mathcal T^*_{n+1} >0$ 
there exists $\mathcal T_{n+1}\geq  \mathcal T^*_{n+1}$ such that 
$$
\frac{1} { {\mathcal T}_{n+1}}  \text{ Leb} \left(\{ T \in [1, {\mathcal T}_{n+1}]  \vert  G(T,\theta) \not\in \mathcal V_n \}\right) < \epsilon_{n+1}\,.
$$
Let $Z_\theta$ be the set defined as follows:
$$
Z_\theta:=  \cup_{n\in \N}   \{ T \in [{\mathcal T}_n,{\mathcal T}_{n+1}] :   G(T,\theta) \not\in \mathcal V_n\} \,.
$$
Let us find under what conditions $Z_\theta$ has zero lower density.  We have
$$
\begin{aligned}
\text{Leb} (Z_\theta \cap [0,\mathcal T_n]) &\leq  \sum_{k=1}^{n-1} \text{Leb} (Z_\theta \cap [{\mathcal T}_k, {\mathcal T}_{k+1}])  \leq   \sum_{k=1}^{n-1} \epsilon_{k+1} {\mathcal T}_{k+1}  \,.
\end{aligned}
$$
It is therefore enough to choose the sequences recursively so that
$$
\frac{1}{\mathcal T_n} \sum_{k=1}^{n-2} \epsilon_{k+1} \mathcal T_{k+1} + \epsilon_n  \to   0\,.
$$
It is clear by the definition of the set $Z_\theta \subset \R$ that, for $\theta \in \mathcal F_I$, we have
$$
\lim_{T\not\in Z_\theta}  G(T,\theta) = \mu\,. 
$$
The argument is therefore complete.

\end{proof}

\begin{proof}[Proof of Corollary \ref{cor:BG2}]

For all $(T, \theta)\in [1, +\infty) \times \T$, let $d(T,\theta)$ denote the distance from the measure $G(T,\theta)$ to $\mu$ with respect to any
metric which induces the weak* topology.  Let us assume by contraposition that there exists a positive measure set $\mathcal P \subset I$
such that, for all $\theta \in \mathcal P$, we have
$$
\limsup_{T\to +\infty} \, d(T,\theta)>0\,.
$$
This implies that there exists $\epsilon >0$ and a set $\mathcal P_\epsilon$ of positive Lebesgue measure such that 
for all $\theta \in \mathcal P_\epsilon$ there exists a sequence $(T_n(\theta))$ such that, for all $n\in \N$,
$$
 d(T_n(\theta),\theta)\geq \epsilon\,.
$$
By a straightforward argument, for any continuous function $f\in C^0(X)$ and for any $h\in \R$, we have
$$
\vert [G(T+h,\theta) -G(T,\theta)] (f)\vert  \leq   2 \frac{h}{T} \vert f \vert_{L^\infty}\,.
$$
It follows that there exists $\delta>0$ such that, for all $T\in [(1-\delta)T_n(\theta), (1+\delta)T_n(\theta)]$ we have
$$
d(T,\theta)\geq \epsilon/2 \,,
$$
hence it is possible to construct a sequence of compactly supported measures $\pi_n$ on $[1, +\infty) \times I$ with smooth bounded density and
conditional measure on $\T$ equal to the Lebesgue measure, such that
$$
\lim_{n\to +\infty} \pi_n \left( \{ (T,\theta) /   d(T, \theta) \geq \epsilon/2\} \right) >0\,.
$$
This contradicts the conclusion of Theorem~\ref{thm:BG}, hence the corollary is proven.

\end{proof}

\section{Oseledets regularity in almost all directions}
\label{sec:OR}

In this section we prove the Oseledets regularity result stated in Theorem \ref{thm:OG}.

\begin{proof} 
Let $\Proj^1(H)$ denote the projectivization of the irreducible bundle $H$ over the separable metric space $X$. Let $x \in X$
satisfying the hypothesis with respect to the $\SL(2, \R)$-invariant measure $\mu$ on $X$. 
Let us recall that there exists a sequence $\{\sigma_\ell\}$ of continuous function  $\sigma_\ell: \Proj^1(H) \to \R$ such that the following holds. For any $g^H_\R$-invariant probability measure $\nu$ on $\Proj^1 (H)$  which projects to the $g_\R$-invariant probability measure $\mu$ on $X$, and for any $\ell \in \N$ we have
\begin{equation}
\label{eq:exp_ineq}
\int_{\Proj^1 (H)}  \sigma_\ell (\rm v)  d\nu(v) \leq  \lambda_\mu.
\end{equation}
If, in addition, the measure $\nu$ is supported on the Oseledets subspace of the top Lyapunov exponent $\lambda_\mu$  of $g^H_\R$ with respect to the measure $\mu$ on $X$, then we have
\begin{equation}
\label{eq:exp_id}
\int_{\Proj^1 (H)}  \sigma_\ell (\rm v)  d\nu(v) =  \lambda_\mu.
\end{equation}
Finally, since the cocycle is uniformly Lipschitz, for any $x\in X$ and any ${\rm v}\in \Proj^1 (H_x)$,  uniformly with respect to $T \geq 1$ we have  
\begin{equation}
\label{eq:exp_int}
  \frac{1}{T} \log \vert g^H_T ({\rm v}) \vert_{g_T(x)}  = \lim_{\ell\to +\infty}   \frac{1}{T}  \int_0^T  \sigma_\ell( g_t^H ({\rm v}) ) \, dt  \,,
\end{equation}
The functions $\sigma_\ell : \Proj^1(H) \to \R$ can be defined as follows: for all $\ell \in \N$, for all $x\in X$ and all ${\rm v}\in \Proj^1(H_x)$, let
$$
\sigma_\ell (\rm v) :=  \ell \log \left( \frac{  \vert g^H_{1/\ell} ({\rm v})\vert_{g_{1/\ell}(x)}} {  \vert  {\rm v} \vert_x} \right)  \,.
$$
It is immediate to verify that, for all $L\in \N$, by telescopic summation  we have
$$
 \log\left( \frac{\vert g^H_L ({\rm v}) \vert_{g_L(x)}}{\vert \rm v \vert_x}\right)  =   \frac{1}{\ell}  \sum_{j=0}^{\ell L-1}   \sigma_\ell ( g^H_{j/\ell} (\rm v))\,.
$$
By the uniform Lipschitz property we also have the estimate
$$
\vert  \frac{1}{\ell L}  \sum_{j=0}^{\ell L-1}   \sigma_\ell ( g^H_{j/\ell} (\rm v)) -  \frac{1}{L} \int_0^{L}   \sigma_\ell ( g^H_t (\rm v))\, dt \vert  \leq  \frac{K}{\ell} \,.
$$
The above claims \eqref{eq:exp_ineq}, \eqref{eq:exp_id} and \eqref{eq:exp_int} follow from Birkhoff ergodic theorem and Oseledets  multiplicative ergodic theorem.

\smallskip
By a result of Bonatti, Eskin and Wilkinson (see \cite{BEW}, Theorem 1.3), under the irreducibility assumption,  any $P$-invariant probability measure $\nu$ on $\Proj^1 (H)$ 
which projects to the $\SL(2, \R)$-invariant probability measure $\mu$ on $X$ is supported on the Oseledets subspace of the top Lyapunov exponent, hence 
identity \eqref{eq:exp_id} holds. 

\smallskip
Let $x\in X$ be any point satisfying the assumption of the theorem. For any  $ {\rm v} \in \Proj^1(H_x)$ let us consider the measures $\nu_n$ on $\Proj^1(H)$ given,
for all $f \in C^0_0(\Proj^1 (H))$ by the formula
$$
\nu_n (f) :=    \int_{S_I} \frac{1}{T}  \int_0^T  f(g^H_t(r_\theta({\rm v})) dt  d\pi_n(T, \theta)  \,.
$$
Let $\nu$ be any weak* limit  point (along a subsequence) of the above family of measures.  The measure $\nu$ is $P$-invariant 
(invariant under the action of the maximal  parabolic subgroup generated by the diagonal subgroup and the unstable unipotent) and, by the hypothesis on $x\in X$,
it projects to $\mu$ under the canonical projection $\Proj^1(H) \to X$, hence identity \eqref{eq:exp_id} holds.

Let us now consider the map ${\mathcal L}: I \to B_1$ from the cylinder $S_I$ to the space of measures on the compactification $\hat \Proj^1(H)$
of the bundle $\Proj^1(H)$, given for every $f \in C^0(\hat \Proj^1(H))$ by the formula
$$
{\mathcal L}(T,\theta) (f)  := \frac{1}{T}  \int_0^T  f( g^H_t (r_\theta (\rm v)) dt \,.
$$
Let $\mathcal L_\infty$ be any weak* limit of the push-forwards ${\mathcal L}_n:= ({\mathcal L})_*(\pi_n)$ under the above maps.  By construction, for all $\ell\in \N$ we have the identity
$$
\begin{aligned}
\int  F_{\sigma_\ell} d{\mathcal L}_\infty &= \lim_{n\to +\infty}  \int  F_{\sigma_\ell}  d{\mathcal L} _n  \\ &= \lim_{n\to +\infty}  \int_{S_I}  \frac{1}{T} 
 \int_0^T  \sigma_\ell (g^{H}_t (r_\theta (\rm v)) dt  d\pi_n(T, \theta)=
\int \sigma_\ell d\nu = \lambda_\mu\,.
\end{aligned}
$$
We claim that ${\mathcal L}_\infty$ is supported on the closed subset $\mathcal C$ of $g^{H}_\R$-invariant probability measures $\nu$ on the sub-bundle  $\Proj^1 (H)$ 
such that
\begin{equation}
\label{eq:Cset}
\int_{\Proj^1 (H)}  \sigma_\ell \, d\nu = \lambda_\mu\,, \quad \text{ for all } \ell\in \N\,.
\end{equation}
In fact, by an argument similar to that of section \S \ref{sec:BG} it can be proved that ${\mathcal L}_\infty$ is supported on the set of all $g^{H}_\R$-invariant measures on $\hat \Proj^1(H)$  which project to $\mu$ under the projection $\Proj^1(H) \to X$ and, since $\lambda_\mu$ is the top Lyapunov exponent,  for all such measures $\nu$ on $\Proj^1(H)$ we have the inequalities in formula \eqref{eq:exp_ineq}, that is, 
$$
\int \sigma_\ell  d\nu \leq \lambda_\mu\,, \quad \text{ for all } \ell\in \N\,.
$$
It follows by definition  that $F_{\sigma_\ell} \leq \lambda_\mu$ on  $\text{supp}({\mathcal L}_\infty)$ and,  since, as proved above,
$$
\int  F_{\sigma_\ell} d{\mathcal L}_\infty  =\lambda_\mu \,, \quad \text{ for all } \ell\in \N\,,
$$
it follows that ${\mathcal L}_\infty$ is supported on the set $\mathcal C$, as claimed.

From the above conclusion we immediately derive that, 
for every neighborhood $\mathcal V$ of the closed subset $\mathcal C$ in the (metric) space $B_1$,  we have
\begin{equation}
\label{eq:Omeasure_bound_1}
\lim_{n\to +\infty}   \pi_n (\{ (T,\theta)\in [1, +\infty) \times I \vert   {\mathcal L}(T,\theta) \not\in \mathcal V\} ) =0\,.
\end{equation}
Since $\pi_n (\{\infty\})=0$, it follows that for every $\epsilon>0$ there exists $T_\epsilon>1$ such that
$$
\pi_n (\{ (T,\theta)\in [T_\epsilon, +\infty) \times I \vert   {\mathcal L}(T,\theta) \not\in \mathcal V\} ) < \epsilon, \quad \text{ for all }n \in\N\,.
$$
Let $(\mathcal V_k)$ be a basis of neighborhoods of $\mu$ in $B_1$ and let $(\epsilon_k)$ be any {\it summable} sequence of positive real numbers. 
There exists an increasing diverging sequence $(T_k)$ such that 
$$
\pi_n (\{ (T,\theta)\in [T_k, +\infty) \times I \vert  {\mathcal L}(T,\theta) \not\in \mathcal V_k\} ) < \epsilon_k, \quad \text{ for all }n \in\N\,.
$$
Let then $\mathcal Z$ be the set such that
$$
{\mathcal Z} \cap [T_k, T_{k+1}) = \{ (T,\theta)\in [T_k, T_{k+1}) \times I \vert   {\mathcal L}(T,\theta) \not\in \mathcal V_k\}\,.
$$
It is clear that by construction all  weak limits of the family $\{ {\mathcal L}(T,\theta) \vert  (T,\theta)\not\in {\mathcal Z}\}$
belong to the closed set $\mathcal C$ (see its definition in formula~\eqref{eq:Cset}), hence, for all $\ell \in \N$,
$$
\lim_{(T,\theta)\not\in {\mathcal Z}} \frac{1}{T} \int_0^T  \sigma_\ell (g_t^H r_\theta {\rm v}) dt    =   \lambda_\mu\,.
$$
Finally,  by the uniform approximation property stated in formula~\eqref{eq:exp_int}, we have
$$
\lim_{(T,\theta)\not\in {\mathcal Z}}   \frac{1}{T} \log \vert g^H_T ({\rm v}) \vert_{g_T(x)} =   \lambda_\mu\,.
$$
It remains to prove that the limit of the sequence $\{\pi_n(Z)\}$ is equal to zero. Since $\pi_n \to \delta^I_\infty$ it follows that, for any $k\in \N$, we have
$$
\lim_{n\to+\infty} \pi_n ( [1, T_k) \times I )=0\,, 
$$
while by construction, for all $n\in \N$, we have
$$
\lim_{k\to+\infty} \pi_n \left( {\mathcal Z} \cap ([T_k, +\infty) \times I)\right) \leq \lim_{k\to+\infty}  \sum_{j\geq k} \epsilon_j =0 \,.
$$
The statement of the theorem follows.
\end{proof}

We conclude the section by proving Corollaries  \ref{cor:OG1}  and \ref{cor:OG2}.

\begin{proof}[Proof of Corollary \ref{cor:OG1}] The argument is similar to the proof of Corollary \ref{cor:BG1}.

Let $(\tau_n)$ be any sequence of probability measures on $[1, +\infty)$ which converges to the Dirac mass at the point at infinity.
Let  $\Theta_I$ denote the normalized Lebesgue measure on the interval $I\subset \T$ and let then $(\pi_n)$ be the sequence of probability measures 
on $[1, +\infty) \times I$ defines as
$$
\pi_n :=   \tau_n \times  \Theta_I  \,, \quad \text{ for all } n\in \N\,.
$$
By the hypothesis of the corollary that
$$
\mu:=\lim_{T\to +\infty} \frac{1}{T}  \int_0^T  \frac{1}{\vert I\vert} \int_I   (g_t\circ r_\theta)_* (\delta_x)d\theta dt 
$$
it follows that the hypothesis of Theorem~\ref{thm:OG} holds for the sequence $(\pi_n)$. 

Therefore there exists a set $\mathcal Z$ such that   $\lim_{n\to +\infty}\pi_n ({\mathcal Z}) =0$
and such that, for all ${\rm v} \in H_x\setminus \{0\}$, we have
$$
\lim_{(t,\theta) \not\in {\mathcal Z}}  \frac{1}{t}  \log  \vert  g^H_t (r_\theta(\rm v) ) \vert_{g_t (r_\theta(x))}  = \lambda_\mu.
$$
From the above conclusion we derive that, 
for every sequence of probability measure $(\tau_n)$ weakly converging to the delta mass at $+\infty\in [1, +\infty]$ and
for every $\delta >0$,  we have
\begin{equation}
\label{eq:Ozero_lim_measure}
\lim_{n\to +\infty}   (\tau_n \times\Theta_I)\left(\{ (t,\theta) \vert   \frac{1}{t}  \log  \vert  g^H_t (r_\theta(\rm v) ) \vert_{g_t (r_\theta(x))}  \not\in  
(\lambda_\mu-\delta, \lambda_\mu+\delta)\} \right) =0\,.
\end{equation}
We claim that, for every $\delta>0$, there exists a full measure set $\mathcal F_{\delta} \subset I$ such that, for all $\theta \in \mathcal F_{\delta}$, we have
$$
\liminf_{T \to +\infty}  \frac{1}{T}  \text{Leb}\left( \{ t\in [1, T] \vert    \frac{1}{t}  \log  \vert  g^H_t (r_\theta(\rm v) ) \vert_{g_t (r_\theta(x))}  \not\in  
(\lambda_\mu-\delta, \lambda_\mu+\delta)\} \right)=0\,.
$$
Otherwise there exists a positive measure set $\mathcal P_\delta \subset I \subset \T$ such that, for all $\theta\in \mathcal P_\delta$,
$$
\liminf_{T \to +\infty}  \frac{1}{T}  \text{Leb}\left( \{ t\in [1, T] \vert    \frac{1}{t}  \log  \vert  g^H_t (r_\theta(\rm v) ) \vert_{g_t (r_\theta(x))}  \not\in  
(\lambda_\mu-\delta, \lambda_\mu+\delta)\} \right)>0\,.
$$
By the Egorov theorem it follows that there exists a  sequence of times $(T_n)$ and a positive measure subset 
$\mathcal P'_\delta  \subset \mathcal P_\delta \subset I$ such that  the sequence
$$
 \frac{1}{T_n}  \text{Leb}\left( \{ t\in [1, T_n] \vert    \frac{1}{t}  \log  \vert  g^H_t (r_\theta(\rm v) ) \vert_{g_t (r_\theta(x))}  \not\in  
(\lambda_\mu-\delta, \lambda_\mu+\delta)\} \right)
$$
converges uniformly to a continuous positive function on $\mathcal P'_\delta$.

 It is then possible to construct a sequence $(\tau_n)$ of 
probability measures on $[1, +\infty)$, weakly converging to the delta mass at $+\infty\in [1, +\infty]$, which contradicts the conclusion in
formula \eqref{eq:Ozero_lim_measure}.  Hence the claim is proved.

For any (fixed) sequence $(\delta_n)$ converging to zero,   let $\mathcal F_I$ denote the full measure set defined as 
$$
\mathcal F_I := \bigcap_{n\in \N} \mathcal F_{\delta_n}\,.
$$
From the above claim it follows that, for all $\theta\in \mathcal F_I$ and for all $n\in \N$, we have 
$$
\liminf_{\mathcal T \to +\infty}  \frac{1}{T}  \text{Leb}\left( \{ t\in [1, T] \vert    \frac{1}{t}  \log  \vert  g^H_t (r_\theta(\rm v) ) \vert_{g_t (r_\theta(x))}  \not\in  
(\lambda_\mu-\delta_n, \lambda_\mu+\delta_n)\} \right) =0\,.
$$
In particular, for any sequence $(\epsilon_n)$ of positive real numbers, converging to zero, we have that there exist increasing diverging sequences
$(T_n) \subset [1, +\infty)$ such that 
$$
\frac{1} { T_n}  \text{ Leb} \left(\{ t \in [1, T_n]  \vert  \frac{1}{t}  \log  \vert  g^H_t (r_\theta(\rm v) ) \vert_{g_t (r_\theta(x))}  \not\in  
(\lambda_\mu-\delta_n, \lambda_\mu+\delta_n)\} \right) < \epsilon_n\,.
$$
Such sequence can be constructed recursively as follows. For any finite increasing sequence $\{T_k\vert k\ \leq n\}$, and for any $T^*_{n+1} >0$ 
there exists $T_{n+1}\geq  T^*_{n+1}$ such that 
$$
\frac{1} { T_{n+1}}  \text{ Leb} \left(\{ t \in [1, T_{n+1}] \vert  \frac{1}{t}  \log  \vert  g^H_t (r_\theta(\rm v) ) \vert_{g_t (r_\theta(x))}  \not\in  
(\lambda_\mu-\delta_n, \lambda_\mu+\delta_n) \}\right) < \epsilon_{n+1}\,.
$$
Let $Z_\theta$ be the set defined as follows:
$$
Z_\theta:=  \cup_{n\in \N}   \{ T \in [{T}_n,{T}_{n+1}] :   \frac{1}{t}  \log  \vert  g^H_t (r_\theta(\rm v) ) \vert_{g_t (r_\theta(x))}  \not\in  
(\lambda_\mu-\delta_n, \lambda_\mu+\delta_n)\} \,.
$$
Let us find under what conditions $Z_\theta$ has zero lower density.  We have
$$
\begin{aligned}
\text{Leb} (Z_\theta \cap [0, T_n]) &\leq  \sum_{k=1}^{n-1} \text{Leb} (Z_\theta \cap [{T}_k, {T}_{k+1}])  \leq  
 \sum_{k=1}^{n-1} \epsilon_{k+1} {T}_{k+1}  \,.
\end{aligned}
$$
It is therefore enough to choose the sequences recursively so that
$$
\frac{1}{T_n} \sum_{k=1}^{n-2} \epsilon_{k+1} T_{k+1} + \epsilon_n  \to   0\,.
$$
It is clear by the definition of the set $Z_\theta \subset \R$ that, for $\theta \in \mathcal F_I$, we have
$$
\lim_{t\not\in Z_\theta}   \frac{1}{t}  \log  \vert  g^H_t (r_\theta(\rm v) ) \vert_{g_t (r_\theta(x))}  = \lambda_\mu\,. 
$$
The argument is therefore complete.
\end{proof}

\begin{proof}[Proof of Corollary \ref{cor:OG2}]
 Let us assume by contraposition that there exists a positive measure set $\mathcal P \subset I$
such that, for all $\theta \in \mathcal P$, we have
$$
\limsup_{t\to +\infty} \, \left\vert \frac{1}{t}  \log  \vert  g^H_t (r_\theta(\rm v) ) \vert_{g_t (r_\theta(x))}  - \lambda_\mu \right\vert \, > \, 0\,.
$$
This implies that there exists $\epsilon >0$ and a set $\mathcal P_\epsilon$ of positive Lebesgue measure such that 
the following holds. For all $\theta \in \mathcal P_\epsilon$ there exists a diverging sequence $(t_n)=(t_n(\theta))$ we have, for all $n\in \N$, 
$$
\left \vert \frac{1}{t_n}  \log  \vert  g^H_{t_n} (r_\theta(\rm v) ) \vert_{g_{t_n} (r_\theta(x))}  - \lambda_\mu \right\vert \geq \epsilon\,.
$$
Since the cocycle is by hypothesis uniformly Lipschitz, there exists $\delta>0$ such that, for all $t\in [(1-\delta)t_n(\theta), (1+\delta)t_n(\theta)]$ we have
$$
\left\vert \frac{1}{t}  \log  \vert  g^H_{t} (r_\theta(\rm v) ) \vert_{g_{t} (r_\theta(x))}  - \lambda_\mu \right\vert \geq \epsilon/2 \,, 
$$
hence it is possible to construct a sequence of compactly supported measures $\pi_n$ on $[1, +\infty) \times I$ with smooth bounded density and
conditional measure on $\T$ equal to the Lebesgue measure, such that
$$
\lim_{n\to +\infty} \pi_n \left( \{ (T,\theta) /  \left\vert \frac{1}{t}  \log  \vert  g^H_{t} (r_\theta(\rm v) ) \vert_{g_{t} (r_\theta(x))}  - \lambda_\mu \right\vert \geq \epsilon/2\} \right) >0\,.
$$
This contradicts the conclusion of Theorem~\ref{thm:BG}, hence the corollary is proven.

\end{proof}

\section{Limits of geodesic push-forwards of horospherical measures}
\label{sec:horospheres}

Let $X$ be a stratum of the moduli space of Abelian differentials. Let $\mathcal H_X$ denote the set of all 
compactly supported probability measures on $X$ supported on a leaf $\mathcal F^s(x)$ of the stable foliation 
of the Teichm\"uller geodesic flow such that the following properties hold:
\begin{enumerate}
\item  the measure is absolutely continuous with continuous density with respect to the canonical affine measure on $\mathcal F^s(x)$;
\item  almost all of its conditional measures along the stable Teichm\"uller horocycle are restrictions of Lebesgue 
measures along horocycle orbits.
\end{enumerate}
In particular, we may consider the restriction of the canonical  affine measure to a compact  subset of a leaf 
of the  stable foliation. 

By the results of Eskin, Mirzakhani and Mohammadi \cite{EM}, \cite{EMM}, and by condition $(2)$ above, we can deduce
that for  any horospherical probability measure $\nu \in \mathcal H_X$ there exists a unique $\SL(2, \R)$-invariant affine ergodic probability measure 
$\mu$ on $X$ such that, in the weak* topology, we have 
$$
\lim_{T\to +\infty} \frac{1}{T}  \int_0^T  (g_t)_* (\nu) dt   = \mu\,.
$$
By the argument explained in Section~\ref{sec:LGP}  we can then deduce that there exists a set $Z\subset \R$ 
of zero upper density such that, in the weak* topology, we have
\begin{equation}
\label{eq:horos_conv_Z}
\lim_{t\not\in Z} (g_t)_* (\nu)   = \mu\,.
\end{equation}
Our goal in this section is to prove Theorem~\ref{thm:horospheres}. 

Let $\Vert \cdot \Vert_{X}$ denote the Hodge norm on the tangent space $TX$ of an (affine) $\SL(2,\mathbb \R)$-invariant suborbifold  $X$
of the moduli space of Abelian differentials and let $d_X:~X\times~X\to\R$ denote the corresponding distance function. Let 
${\rm Lip}(X)$ denote the space of Lipschitz continuous functions with respect to the metric $d_X$ on $X$  endowed with the norm
$$
\Vert f \Vert_{\rm Lip} := \vert f \vert_{C^0(X)} + \sup_{x\not =y} \frac{ \vert f(x) -f(y)\vert}{d_X(x,y)} \,, \quad \text{ for all } f \in {\rm Lip}(X)\,.
$$
We recall that by Ascoli-Arzel\`a theorem, for any compact set $K\subset K$ a  ball  ${\rm Lip}(X,R)$ of radius $R>0$ in ${\rm Lip}(X)$
maps under the restriction map $R_K :C^0(X) \to C^0(K)$ into a compact subset. 

Let $\mathcal F^s$ denote the strong stable foliation of the (Teichm\"uller) geodesic flow.  For all $x \in X$, let $D^s(x,r) \subset \mathcal F^s(x)$  denote the
stable disk 
$$
D^s(x,r) :=  \{ y \in \mathcal F^s (x) \vert   d_X(x,y) \leq r\}\,.
$$
Let ${\mathcal I}^s_r: C^0(X) \to C^0(X)$ denote the averaging operator  along the stable disks with respect to the Hodge volume  $\text{vol}^s$ on stable leaves, that is, 
$$
{\mathcal I}^s_r (f) (x) :=   \frac{1}{{\rm vol} (D^s(x,r))} \int_{D^s(x,r)}   f  d{\rm vol}^s \,, \quad \text{ for all } f \in C^0(X) \,.
$$
Let $\mathcal F^{wu}$ denote the weak-unstable foliation of the geodesic flow. Let ${\rm Lip}^{wu}(X)$ denote  the space 
 of continuous functions which are Lipschitz along the weak-unstable foliation, that is,
$$
{\rm Lip}^{wu}(X):= \{ f\in C^0(X) \vert  \sup_{x\in X}\sup_{ y\in \mathcal F^{wu}(x) }    \frac{ \vert f(x) -f(y)\vert}{d_X(x,y)}  < +\infty\}\,,
$$
endowed with the norm 
$$
\Vert f \Vert_{{\rm Lip}^{wu}} := \vert f \vert_{C^0(X)}  + \sup_{x\in X}\sup_{ y\in \mathcal F^{wu}(x) }    \frac{ \vert f(x) -f(y)\vert}{d_X(x,y)} 
\,, \quad \text{ for all } f \in {\rm Lip}^{wu}(X)\,.
$$
We have the following immediate result.

\begin{lemma} 
\label{lemma:compactness}
For every $r>0$, the averaging  operator  ${\mathcal I}^s_r$  maps ${\rm Lip}^{wu}(X)$ continuously into 
${\rm Lip}(X)$, hence for any compact set $K\subset X$ the composition 
$$
{\mathcal I}^s_r \circ R_K : {\rm Lip}^{wu}(X)  \to  C^0_c(X)
$$
is a compact operator. 
\end{lemma} 
\begin{proof} We can prove  by an immediate estimate  that the averaging map $ {\mathcal I}^s_r$ maps the Banach space ${\rm Lip}^{wu}(X)$  
continuously to the Banach space ${\rm Lip}(X)$, and for any compact set $K\subset X$ it maps ${\rm Lip}^{wu}(K): = {\rm Lip}^{wu}(X)\cap C^0(K)$
into ${\rm Lip}(K):= {\rm Lip}(X)\cap C^0(K) $. By Ascoli-Arzel\`a theorem, the embedding ${\rm Lip}(K)$ into $C^0(K)$  is a compact operator. Finally, 
the composition of a continuous (bounded) operator and a compact operator is a compact operator.
\end{proof}

We finally prove the convergence of push-forwards of horospherical measures.
\begin{proof}[Proof of Theorem~\ref{thm:horospheres}]
By Lemma~\ref{lemma:compactness} the operator ${\mathcal I}^s_r \circ R_K : {\rm Lip}^{wu}(X)  \to  C^0_c(X)$ is compact, hence  the dual operator 
$$( {\mathcal I}^s_r \circ R_K)^*: C^0(K)^* \to {\rm Lip}^{wu}(X)^*$$ from the space $ C^0_c(X)^*$ of linear 
continuous functionals on $C^0_c(X)$ to the space ${\rm Lip}^{wu}(X)^*$ of linear continuous functionals on ${\rm Lip}^{wu}(X)$ is also compact. In particular,
for any weakly converging sequence $(\nu_n) \subset \mathcal M(X)$ of probability measures on $X$, the sequence  $ R^*_K ({\mathcal I}^s_r)^* (\nu_n)$ 
is (strongly) convergent in ${\rm Lip}^{wu}(X)^*$.  By construction, we have that for all $t\geq 0$ the pull-back operator
$(g_{-t})^* :  {\rm Lip}^{wu}(X) \to {\rm Lip}^{wu}(X)$ is a weak contraction, in the sense that
$$
\Vert  f \circ g_{-t}  \Vert_{{\rm Lip}^{wu}} \leq   \Vert  f  \Vert_{{\rm Lip}^{wu}} \,, \quad \text{ for all }  f \in {\rm Lip}^{wu}(X)\,.
$$ 
hence the dual operator $(g_t)_* : {\rm Lip}^{wu}(X)^* \to {\rm Lip}^{wu}(X)^*$ defined as
$$
(g_t)_* (\nu) (f) :=  \nu ( f \circ g_{-t}) \,, \quad \text{ for all } \nu \in {\rm Lip}^{wu}(X)^* \text{ and for all } f\in {\rm Lip}^{wu}(X)
$$
is also a weak contraction with respect to the dual norm $\Vert \cdot \Vert^*_{{\rm Lip}^{wu}}$ on ${\rm Lip}^{wu}(X)^*$:
$$
\Vert  (g_t)_*(\nu) \Vert^*_{{\rm Lip}^{wu}} \leq   \Vert  \nu  \Vert^*_{{\rm Lip}^{wu}} \,, \quad \text{ for all }  \nu \in {\rm Lip}^{wu}(X)^*\,.
$$
Let $\nu$ be any horospherical measure supported on the stable leaf $\mathcal F^s(x)$ at a point $x\in X$. Let $\mu$ denote the
unique affine probability measure supported on the orbit closure $\overline{\SL(2,\R) x}$. As we have remarked above, see formula
\eqref{eq:horos_conv_Z}, there exists a sequence $(t_n)$ such that 
$$
(g_{t_n})_*(\nu) \to \mu   \quad \text{ in the weak* topology}\,.
$$
Since $\nu$ is a {\it horospherical measure}, it follows that 
$$
\lim_{t\to +\infty} \Vert  ({\mathcal I}^s_r)^* (g_{t_n})_*(\nu) - (g_{t_n})_*(\nu) \Vert^*_{{\rm Lip}^{wu}}   =0\,,
$$
hence, for any compact set $K \subset X$, we have
$$
\lim_{n\to +\infty} \Vert   R^*_K  (g_{t_n})_* (\nu) - R^*_K (\mu) \Vert^*_{{\rm Lip}^{wu}}  = 0\,.
$$
Since $(g_t)_*$ is a weak contraction on ${\rm Lip}^{wu}(X)^*$, uniformly with respect to $t\geq 0$, we have that
$$
\lim_{n\to +\infty} \Vert  (g_t)_*  R^*_K (g_{t_n})_* (\nu) -  R^*_K (\mu)\Vert_ {{\rm Lip}^{wu}}  = 0\,.
$$
Finally, since the set of all probability measures supported on horocycle arcs is {\it tight} (see \cite{MW}, \cite{EMa})
for any $\epsilon>0$ there exists a compact set $K_\epsilon$ such that
$$
\Vert   (g_t)_*  R^*_{K_\epsilon} (g_{t_n})_* (\nu) - (g_{t+t_n})_* (\nu)\Vert_{\mathcal M(X)} =
\Vert  R^*_{K_\epsilon} (g_{t_n})_* (\nu) - (g_{t_n})_* (\nu)\Vert_{\mathcal M(X)}  \leq \epsilon \,,
$$
hence the measure $\mu$ is the unique weak* limit of the set  $\{(g_{t+t_n})_* (\nu)\}$, as claimed. 

\end{proof}

 \end{document}